\documentclass[a4paper]{article}
\usepackage{amsthm,amsfonts,amsmath,amssymb}
\usepackage[abbrev,nobysame]{amsrefs}
\usepackage[cp1251]{inputenc}
\usepackage[english]{babel}
\usepackage[final]{graphicx}
\usepackage{setspace}
\usepackage[12pt]{extsizes}
\begin{document}
%\large
\newcommand{\per}{{\rm per}}
\newcommand{\Per}{{\rm Per}}
\newtheorem{teorema}{Theorem}
\newtheorem{lemma}{Lemma}
\newtheorem{utv}{Proposition}
\newtheorem{svoistvo}{Property}
\newtheorem{sled}{Corollary}
\newtheorem{con}{Conjecture}
\newtheorem{zam}{Remark}
\newtheorem{quest}{Question}

\author{A. A. Taranenko\thanks{Sobolev Institute of Mathematics, Novosibirsk, Russia. 
 \texttt{taa@math.nsc.ru}}}
\title{Products of multidimensional matrices, stochastic matrices, and permanents}
\date{March 30, 2023}

\maketitle

\begin{abstract}
In this paper we consider four basic multidimensional matrix operations (outer product, Kronecker product, contraction, and projection)  and two derivative operations (dot and circle products).  We start with the interrelations between these operations and deduce some of their algebraic properties. Next, we study their action on $k$-stochastic matrices. At last, we prove several  relations on the permanents of products of multidimensional matrices. In particular, we obtain that the permanent of the dot product of nonnegative multidimensional matrices is not less than the product of their permanents and show that  inequalities on the Kronecker product of nonnegative 2-dimensional matrices  cannot be extended to the multidimensional case. 


\end{abstract}

\section{Introduction and basic definitions}

Let $n, d \in \mathbb{N}$ and $I_n^d = \{  (\alpha_1, \ldots, \alpha_d) | \alpha_i \in \{ 1, \ldots, n \}\}$ be the \textit{index set}. For indices  $\alpha \in I_n^{d_1}$, $\beta \in I_n^{d_2}$, let a concatenation $\alpha \beta$ be the index  $\gamma \in I_n^{d_1 + d_2}$ such that $\gamma = (\gamma_1, \ldots, \gamma_{d_1 + d_2}) = (\alpha_1, \ldots, \alpha_{d_1}, \beta_1, \ldots, \beta_{d_2})$.  

A \textit{$d$-dimensional matrix $A$ of order $n$} is an array $(a_\alpha)_{\alpha \in I^d_n}$, $a_\alpha \in\mathbb R$.   Sometimes $d$-dimensional matrices of order $n$ are considered as tensors of dimension $n$ and order $d$. 

Let $k\in \left\{0,\ldots,d\right\}$. A \textit{$k$-dimensional plane} $\Gamma$ in $A$ is a submatrix of $A$ obtained by fixing $d-k$ components of indices and letting the other $k$ components vary from 1 to $n$. A $1$-dimensional plane is said to be a \textit{line}, and a $(d-1)$-dimensional plane is a \textit{hyperplane}. A \textit{direction} of  a plane $\Gamma$ in the matrix $A$ is a $(0,1)$-vector whose $i$-th component is equal to $1$ whenever the $i$-th component in $\Gamma$ is fixed. 

Let a \textit{transpose} of a matrix $A$ be a permutation of directions of its hyperplanes (permutation of components of all indices). 
We will say that matrices $A$ and $B$ of the same order and dimension are \textit{equivalent} if one can be turned into the other by transposes and  permutations of hyperplanes of the same direction.

One can find many various multidimensional matrix operations and tensor products in the literature, for example,~\cite{CheCveWei.tesorfunc, KilMar.facttensor, shao.tensprod, YaoLiuBu.tensorprod}. We do not aim to observe all of them here. Instead, we focus only on the  properties of four operations: outer and Kronecker products, projection, and contraction. We suppose that these operations  are enough to express all other reasonable multidimensional matrix transformations. As an illustration, we consider the dot product and circle product of multidimensional matrices and present them as an appropriate composition of outer products and contractions. 

Another aim of the present paper is to study the action of matrix products on a set of stochastic matrices and their connections with the multidimensional permanent.

If all $a_\alpha  \geq 0$, a multidimensional matrix $A$ is said to be \textit{nonnegative}.  
A nonnegative $d$-dimensional matrix $A$ of order $n$ is called \textit{$k$-stochastic} if  the sum of entries in each $k$-dimensional plane equals $1$. It is easy to see that if $A$ is a $k$-stochastic matrix of order $n$ and $k < d$, then $\frac{1}{n} A$ is a $(k+1)$-stochastic matrix.  $2$-dimensional $1$-stochastic matrices are known as \textit{doubly stochastic},  multidimensional $1$-stochastic matrices are called \textit{polystochastic}, and polystochastic $(0,1)$-matrices are said to be \textit{multidimensional permutations}. We also denote by $J_n^d$ the $d$-dimensional polystochastic matrix of order $n$, whose all entries are equal to $1/n$. 

Multidimensional stochastic matrices are closely related to latin hypercubes and orthogonal arrays. 
A \textit{$d$-dimensional latin hypercube} $Q$ of order $n$ is a $d$-dimensional matrix of order $n$ such that its entries $q_\alpha$  take values from the set  $\{1, \ldots, n \}$ and  in each line  of $Q$ all $n$ symbols occur. $2$-dimensional latin hypercubes are known as \textit{latin squares}.  A \textit{$t-(n,k,\lambda)$ orthogonal array} is a rectangular $\lambda n^t \times  k$ array $R$  whose entries are chosen from a set $I_n = \{ 1, \ldots, n \}$ such that in every subset of $t$ columns of the array, every $t$-tuple of  elements of  $I_n$ appears in exactly $\lambda$ rows.

The correspondence between a $d$-dimensional latin hypercube $Q$ and a $(d+1)$-dimensional permutation $M(Q)$ is given by the following rule: an entry $q_{\alpha_1, \ldots, \alpha_{d}}$ of a latin hypercube $Q$ equals $\alpha_{d+1}$ if and only if an entry $m_{\alpha_1, \ldots, \alpha_{d+1}}$ of the permutation matrix $M(Q)$ equals 1.  There is also a correspondence between a $t-(n,k,\lambda)$ orthogonal array $R$ and a $k$-dimensional $t$-stochastic matrix $\frac{1}{\lambda} M$ of order $n$: an entry $m_{\alpha_1, \ldots, \alpha_k}$ of the matrix $M$ is equal to the number of appearances of the row $(\alpha_1, \ldots, \alpha_k)$ in the array $R$.

To define the permanent of a multidimensional matrix, we need to define the set of its diagonals. 
A \textit{diagonal} $D$ in a $d$-dimensional matrix of order $n$ is a collection of indices $\{  \alpha^1, \ldots, \alpha^n \}$ such that  $\alpha^i_k \neq \alpha^j_k$ for all $k \in \{  1, \ldots, d\}$ and all  $i,j \in \{1, \ldots, n\}$, $i \neq j$.  A diagonal  
$$\{  (1, \ldots, 1), \ldots, (n,\ldots, n) \}$$
is said to be \textit{main}.   For a nonnegative matrix $A$, we will say that a diagonal $D$ is \textit{positive} if $a_{\alpha} > 0$ for all $\alpha \in D$. 

 Let $\mathcal{D}(A)$ denote the set of all diagonals in the matrix $A$.
The \textit{permanent} of a $d$-dimensional matrix $A$ of order $n$ is
$$\per A = \sum\limits_{D \in \mathcal{D}(A)} \prod\limits_{\alpha \in D} a_{\alpha}.$$

A diagonal $D$ in a latin hypercube $Q$ is said to be a \textit{transversal}  if all  $q_\alpha$, $\alpha \in D$ are  different.  It is easy to see that every transversal in a latin hypercube $Q$  is a positive diagonal in the $(d+1)$-dimensional permutation matrix $M(Q)$, so the number of transversals in $Q$ is equal to the permanent of  $M(Q)$.

The well-known Birkhoff theorem states that the permanent of every doubly stochastic matrix is positive. For dimensions $d$ greater than $2$, there exist $d$-dimensional polystochastic matrices with zero permanent. The simplest examples of such matrices are multidimensional permutations of odd dimensions $d$ which correspond to iterated groups $\mathbb{Z}_n$ of even order $n$~\cite{wanless.surv}. Since for other $d$ and $n$ we still do not have examples of $d$-dimensional polystochastic matrices of order $n$ with zero permanent,  the following conjecture was proposed in~\cite{my.obzor}.

\begin{con}[\cite{my.obzor}] \label{polycon}
The permanent of every polystochastic matrix of odd order or even dimension is greater than zero.
\end{con}

Since the permanent of a multidimensional permutation is equal to the number of transversals in the corresponding latin hypercube, Conjecture~\ref{polycon} generalizes the well-known Ryser's conjecture on existence transversals in latin squares of odd order~\cite{ryser.conj} and the conjecture on existence transversals in latin hypercubes of odd dimension or odd order by Wanless~\cite{wanless.surv}.

For more information on the properties and applications of stochastic matrices and multidimensional permanents  to  combinatorial problems see~\cite{my.obzor, my.per44}. 

The remaining paper is organized as follows. In Section~\ref{operationsec}, we introduce four basic multidimensional matrix operations (outer and Kronecker products, projection, and contraction)  and two derived operations (dot and circle products).  We also provide many basic algebraic properties and relations for compositions of these operations.  Moreover, in Section~\ref{applsec} we discuss the application of the dot,  Kronecker, and circle products to the  composition and direct product  of the multiary quasigroups,  eigenvalues and coverings of multidimensional matrices.

Section~\ref{stocprodsec} is devoted to the products of stochastic matrices. For given stochastic matrices $A$ and $B$ we establish the degree of stochastivity of their outer, Kronecker, dot, and circle products, as well as their contractions and projections.  In particular, we prove that the Kronecker, dot, and circle product of polystochastic matrices is a polystochastic matrix and fix an error in the theorem on the dot product of stochastic matrices from~\cite{ChildWan.multper}.

At last, in Section~\ref{perprodsec} we study the permanent of products of multidimensional matrices.  We prove that the permanent of the outer product of matrices $A$ and $B$ of order $n$ is the  $n!$ times product of the permanents of $A$ and $B$. Although for the permanent of contractions there are no good estimations, we prove that (as in the 2-dimensional case) the dot product of nonnegative matrices is not less than the product of their permanents. Finally, we consider the permanent of the Kronecker product of nonnegative matrices and show that the upper and lower bounds  for $2$-dimensional matrices from~\cite{brualdi.perdirprod} cannot be extended to the multidimensional case.

\section{Operations on multidimensional matrices} \label{operationsec}

We start with four basic operations for multidimensional matrices and their properties. We give proofs only for several of them  because most of the equalities follow   from  definitions by direct calculation.  We also assume everywhere that  matrices have  orders and dimensions such that all operations are well defined.

\textbf{\textit{Outer product.}}
Let $A$ be a $d_1$-dimensional matrix of order $n$ and $B$  be a $d_2$-dimensional matrix of order $n$.  Then the \textit{outer product} $A\times B$ of matrices $A$ and $B$ is the $(d_1 + d_2)$-dimensional matrix $C$ of the same order $n$ with entries $c_{\alpha \beta} = a_\alpha b_\beta$ for all $\alpha \in I_n^{d_1}$, $\beta \in I_{n}^{d_2}$.

\begin{utv}[Properties of the outer product] \label{outerprop}
\item
\begin{enumerate}
\item $(A \times B) \times C = A \times (B \times C)$.
\item $A \times B$  is obtained from $B \times A$ by transposes of the first and last directions of hyperplanes.
\item $(A + B) \times C = A \times C + B \times C$.
\item $(\lambda A) \times B = \lambda (A \times B) = A \times (\lambda B)$ for every $\lambda \in \mathbb{R}$.
\end{enumerate}
\end{utv}

\textbf{\textit{Kronecker product.}}   Let $A$ be a $d$-dimensional matrix of order $n_1$ and $B$ be a $d$-dimensional  matrix of order $n_2$.  Then the \textit{Kronecker product} $A\otimes B$ of matrices $A$ and $B$ is the $d$-dimensional matrix $C$ of order $n_1 n_2$ with entries $c_{\gamma} = a_\alpha b_\beta,$
where  $\gamma_i = (\alpha_i - 1) n_2 + \beta_i$ for each $i = 1, \ldots, d$. Since the Kronecker product of multidimensional matrices is a natural generalization from the $2$-dimensional case, it has similar properties.

\begin{utv}[Properties of the Kronecker product] \label{kroneckerprop}
\item 
\begin{enumerate}
\item $(A \otimes B) \otimes C = A \otimes (B \otimes C)$.
\item $A \otimes B$  can be obtained from $ B \otimes A$ by permutations of hyperplanes of the same direction.
\item $(A + B) \otimes C = A \otimes C + B \otimes C$.
\item $(\lambda A) \otimes B = \lambda (A \otimes B) = A \otimes (\lambda B)$ for every $\lambda \in \mathbb{R}$.
\end{enumerate}
\end{utv}

\textbf{\textit{Contraction.}}  For a $d$-dimensional matrix $A$ of order $n$ and a set $S = \{ i_1, \ldots, i_\ell \}$, $i_j \in \{ 1, \ldots, d\}$,  let the \textit{$S$-contraction} $A_S$ of the matrix $A$ be the $(d-\ell)$-dimensional matrix $B$ of the same order with entries 
$$b_\beta = \sum\limits_{i=1}^n  a_{\beta_1, \ldots,  i, \ldots, i, \ldots, \beta_{d-\ell} },$$
where  components $i$ are located exactly at positions $i_j \in S$.
In other words, the matrix $A_S$ is obtained from $A$ by a summation of entries over the main diagonal in all $\ell$-dimensional planes  such that their varying components of indices are given by the set $S$. If $S = \{ i,j\}$, we will write \textit{$(i,j)$-contraction} instead of $S$-contraction, and if $S = \{ i\}$, we will write \textit{$i$-contraction}. 

For $2$-dimensional matrices, the $(1,2)$-contraction  is  known as the trace of the matrix. The tensor contraction  is an $(i,j)$-contraction for appropriate $i$ and $j$.

It is possible to further generalize the definition of the contraction and take a summation over an arbitrary (not necessarily main) diagonal in $\ell$-dimensional planes. But we will not study it here because it is the $S$-contraction of a matrix equivalent to the given one and so  it has similar properties. 

Sometimes it is helpful to apply more than one contraction to  the same matrix. Assume that $S_1, \ldots, S_m$ is a collection of pairwise disjoint subsets of $\{ 1, \ldots, d\}$. Let  $A_{S_1; \ldots ; S_m}$ denote  consecutive contraction $(\cdots (A_{S_1})_{S_2} \cdots)_{S_m}$, in which we remunerate the components of indices of matrices only after the last $S_m$-contraction.

\begin{utv}[Properties of the contraction] \label{contractprop}
\item 
\begin{enumerate}
\item If $S \cap T = \emptyset$, then $A_{S ; T} = A_{T; S}$.
\item $(A + B)_S = A_S+ B_S$.
\item $(\lambda A)_S = \lambda A_S $ for every $\lambda \in \mathbb{R}$.
\end{enumerate}
\end{utv}

\textbf{\textit{Projection.}}  For a $d$-dimensional matrix $A$ of order $n$ and a set $S = \{ i_1, \ldots, i_\ell \}$, $i_k \in \{ 1, \ldots, d\}$,  define the \textit{$S$-projection}  $P_S(A)$ of a matrix $A$ to be the $(d-\ell)$-dimensional matrix $B$ of the same order with entries 
$$b_\beta = \sum\limits_{j_1, \ldots, j_\ell=1}^n  a_{\beta_1, \ldots,  j_1, \ldots, j_\ell, \ldots, \beta_{d-\ell} },$$
where components $j_k$ are located exactly at positions $i_k \in S$.
In other words, the matrix $P_S(A)$ is obtained from $A$ by a summation of all entries of $\ell$-dimensional planes in which varying components of indices are given by the set $S$. If $S = \{ i\}$, we will write \textit{$i$-projection} instead of $\{ i\}$-projection. Note that  $i$-contraction and  $i$-projection are the same operation. 

Given  pairwise disjoint sets $S_1, \ldots, S_m$, we define the consecutive projection $P_{S_1; \ldots; S_m}(A)$  as $P_{S_m}(\cdots P_{S_2}(P_{S_1}(A)) \cdots)$, in which we remunerate the components of indices only after the last $S_m$-projection.

\begin{utv}[Properties of the $S$-projection] \label{projectprop}
\item 
\begin{enumerate}
\item If $S \cap T = \emptyset$, then $P_{S;T}(A) = P_{T;S}(A)$.
\item $P_S(A + B) = P_S(A)+ P_S(B)$.
\item $P_S(\lambda A) = \lambda P_S(A) $ for every $\lambda \in \mathbb{R}$.
\end{enumerate}
\end{utv}

The $S$-contraction and $S$-projection can be considered as  reverse operations to the outer product in the following sense.

\begin{utv}
Let $A$ be the $d$-dimensional matrix of order $n$.
\begin{enumerate}
\item If  $S = \{ 1, \ldots, \ell\}$, then $(J_n^{\ell} \times A)_S = A$.
\item If  $S = \{ 1, \ldots, \ell\}$, then  $P_S (J_n^{\ell} \times A) = n^{\ell -1} A$.
\end{enumerate}
\end{utv}

Let us turn to other multidimensional matrix operations that can be obtained as a composition of the above.  We start with a pair of trivial operations.

Suppose that  a $d$-dimensional matrix $A$ is composed of $k$-dimensional matrices $B^i$ in such a way that $B^i$  are  all $k$-dimensional planes in $A$ of a given direction. Then the matrix $A$ can be presented as the sum of the outer products of $B_i$ and appropriate multidimensional $(0,1)$-matrices. 

On the other hand,  for a given multidimensional matrix $A$  one can take its $k$-dimensional plane $\Gamma$ as a new $k$-dimensional matrix. This operation can be expressed as the outer product of $A$ by an appropriate $(0,1)$-matrix and the consecutive contractions of the result. 

Let us consider two more multidimensional matrix products.

\textbf{\textit{Dot product.}}  Let $A$  be a $d_1$-dimensional matrix of order $n$  and $B$ be  a $d_2$-dimensional matrix of order $n$. We define the \textit{$(i,j)$-dot product} $A \cdot_{i,j} B$ of matrices $A$ and $B$ to be the $(d_1 + d_2 - 2)$-dimensional matrix  obtained as the $(i,d_1+ j)$-contraction of the outer product $A \times  B$. For example, if $C$ is a $(d_1, 1)$-dot product of matrices $A$ and $B$, then $c_\gamma = \sum\limits_{i=1}^n a_{\alpha i} b_{i \beta}$, where $\gamma = \alpha \beta$ is the concatenation of indices $\alpha$ and $\beta$. For shortness, we denote the $(d_1, 1)$-dot product of matrices $A$ and $B$ as $AB$ or $A \cdot B$ and call it  \textit{dot product}. Note that the standard dot product $AB$ of $2$-dimensional matrices coincides with our dot product because it is the $(2,1)$-dot product, where the first component of a $2$-dimensional index  is used for rows and the second component labels columns. 

This notion of the dot product can be generalized to a larger collection of matrices.  Suppose that  $A^1, \ldots, A^\ell$ are matrices  of the same order $n$ and dimensions $d_1, \ldots, d_\ell$, respectively. Let $S = (i_1, \ldots, i_\ell)$ be a tuple such that $i_j \in \{ 1, \ldots, d_j\}$. Define the \textit{$S$-dot product} $[A^1_{i_1}, \ldots, A^\ell_{i_\ell}]$ to be  the $(d_1 + \cdots + d_\ell - \ell)$-dimensional matrix  obtained as the $\{ i_1, d_1 + i_2, \ldots, d_1 + \cdots + d_{\ell-1} + i_\ell\}$-contraction of their outer product $A^1 \times \cdots \times A^\ell$. In particular, the standard dot product $AB$ of $2$-dimensional matrices is the matrix $[A_2, B_1]$.

Let us provide some properties of the dot product. Most of them can be easily derived from the ones for the outer product and contraction. They remain essentially the same for a general $(i,j)$-dot product or for the $S$-dot product of more than two matrices.

\begin{utv}[Properties of the dot product] \label{dotprop}
\item 
\begin{enumerate}
\item $(A B) C = A  (B  C)$.
\item $(A + B)  C = A  C + B  C$ and $A (B + C) = AB + BC$. 
\item $(\lambda A)  B = \lambda (A  B) = A  (\lambda B)$ for every $\lambda \in \mathbb{R}$.
\item If  $I$ is the $2$-dimensional identity matrix, then  $A I = I A = A$.
\end{enumerate}
\end{utv}

As in the $2$-dimensional case,  $A B \neq B A$  in general, and $BA$ cannot be obtained from $AB$ by equivalent transformations.

\textbf{\textit{Circle product.}}  Let $A$ be a $d_1$-dimensional matrix of order $n$ and $B$ be a $d_2$-dimensional matrix of order $n$. Define the \textit{circle product} $A \circ B$ to be the $((d_1 -1) (d_2 -1) +1)$-dimensional  matrix of order $n$ equal to the following $(d_1 -1)$-times dot product of $A$ and $B$:
 $$A \circ B = (\cdots((A \cdot_{d_1, 1}B) \cdot_{d_1 - 1, 1}B) \cdots) \cdot_{2, 1}B.$$
 Equivalently, if $C = A \circ B$, then for entries $c_{\gamma}$ we have
 $$c_{\gamma} = \sum\limits_{j_2, \ldots, j_{d_1} =1}^n a_{i, j_2, \ldots, j_{d_1}} b_{j_2 \beta_2} \cdots b_{j_{d_1} \beta_{d_1}}, $$
 where $\gamma = i \beta_2 \cdots \beta_{d_1}$, $i \in \{ 1, \ldots, n\}$, $\beta_{j} \in I_n^{d_2-1}$.

The following properties of  the circle product can be found in~\cite{shao.tensprod} or derived from the definition and properties of the dot product.

\begin{utv}[Properties of the circle product] \label{circprop}
\item 
\begin{enumerate}
\item $(A  \circ B) \circ C = A \circ  (B  \circ C)$.
\item $(A + B) \circ  C = A \circ C + B \circ C$.
\item If $A$ is a $2$-dimensional matrix, then $A \circ (B + C) = A \circ B + A \circ C$.
\item $(\lambda A) \circ  B = \lambda (A \circ  B) $ for every $\lambda \in \mathbb{R}$.
\item $A \circ  (\lambda B) = \lambda^{d_1-1} (A \circ  B) $ for every $\lambda \in \mathbb{R}$, where $d_1$ is the dimension of the matrix $A$.
\end{enumerate}
\end{utv}

We also note that in general $A  \circ B \neq B \circ  A$ because these matrices can even have different dimensions. 

At the end of this section we provide several relations between the defined multidimensional matrix  operations and prove  some of them.

\begin{utv}[Properties of multidimensional operations]
\item 
\begin{enumerate}
\item $(A \times B) \otimes (C \times D) = (A \otimes C) \times (B \otimes D) $.
\item $A_S \times B = (A \times B)_S$.
\item $P_S (A) \times B = P_S(A \times B)$.
\item $( A\cdot_{i,j} B) \times C = (A \times C) \cdot_{i,j} (B \times C)$.
\item $A_S \otimes B_S \ = (A \otimes B)_S$.
\item $P_S(A \otimes B) = P_S (A) \otimes P_S(B)$.
\item $(A \otimes B) \cdot (C \otimes D) = (A \cdot C) \otimes (B \cdot D)$.
\item If $S \cap T = \emptyset$, then $P_T(A_S) = (P_T(A))_S$.
\item If $A$ is a $d$-dimensional matrix and $ d \not\in  S $, then $A_S \cdot B = (A \cdot B)_S$.
\item If $A$ is a $d$-dimensional matrix and $ d \not\in  S $, then $P_S(A) \cdot B = P_S(A \cdot B)$.
\end{enumerate}
\end{utv}

\begin{proof}
1. In order for all operations to be well defined, we require that $A$ be a $d_1$-dimensional matrix of order $n_1$, $B$ be a $d_2$-dimensional matrix of order $n_1$, $C$ be a $d_1$-dimensional matrix of order $n_2$, and $D$ be a $d_2$-dimensional matrix of order $n_2$. 
 
Note that  an entry with index $\mu$ on both sides of the equality is equal to $a_{\alpha} b_\beta c_\gamma d_\delta$, where components  $\mu_i =(\alpha_i -1)n_2 + \gamma_i $ if $1 \leq i \leq d_1$ and  $\mu_i =(\beta_i -1)n_2 + \delta_i $  if $d_1 +1 \leq i \leq d_1 +d_2$.

5. Let $A$ be a $d$-dimensional matrix of order $n_1$ and $B$ be a $d$-dimensional matrix of order $n_2$. Without loss of generality, assume that $S = \{1, \ldots, \ell \}$, $1 \leq \ell \leq d$. Then an entry $c_{\gamma}$ of matrices on both sides of the equation is equal to 
$$\left(\sum\limits_{i=1}^{n_1}  a_{i, \ldots, i, \alpha}\right) \left(\sum\limits_{j=1}^{n_2}   b_{j, \ldots, j, \beta}\right)  = \sum\limits_{i=1}^{n_1} \sum\limits_{j=1}^{n_2}  a_{i, \ldots, i, \alpha}  b_{j, \ldots, j, \beta},$$
 where $\gamma_k = (\alpha_{\ell + k} - 1)n_2 + \beta_{\ell + k}$ for every $k = 1, \ldots, d- \ell$. 

7. In order for all operations to be well defined, we require that $A$ be a $d_1$-dimensional matrix of order $n_1$, $B$ be a $d_1$-dimensional matrix of order $n_2$, $C$ be a $d_2$-dimensional matrix of order $n_1$, and $D$ be a $d_2$-dimensional matrix of order $n_2$. The proof of this equation is essentially the same as for $2$-dimensional matrices. 
\end{proof}

\subsection{Applications to eigenvalues, coverings and multiary quasigroups} \label{applsec}

The notion of the circle product is motivated by  eigenvalues and eigenvectors of multidimensional matrices and hypergraphs. Their study was initiated in \cite{lim.eigentensor, qi.eigentensor} and has been continued in many subsequent papers.

For a $d$-dimensional matrix $A$ of order $n$,  a number $\lambda\in \mathbb{C}$ is called an \textit{eigenvalue} and a vector $v \in \mathbb{C}^n$ is the corresponding \textit{eigenvector} if $A \circ v = \lambda (\mathbb{I} \circ v)$. Here $\mathbb{I}$ is the $d$-dimensional identity $(0,1)$-matrix of order $n$ that  has ones only at the main diagonal.  

For example, let us show that $(d-1)$-stochastic $d$-dimensional matrices have eigenvalue $1$.

\begin{utv} \label{eigenstoch}
Let $A$ be a $d$-dimensional $(d-1)$-stochastic matrix of order $n$. Then $1$ is an eigenvalue of the matrix $A$ corresponding to the eigenvector $e = (1, \ldots, 1)$. 
\end{utv} 

\begin{proof}
By the definition of circle product, the vector $v = A \circ e$ has components $v_i$ equal to the sums of entries in hyperplanes $\Gamma_i$ of the first  direction. Since $A$ is $(d-1)$-stochastic, we have $A \circ e = (1, \ldots, 1) = e$. On the other hand, it holds $\mathbb{I} \circ e = e$. So $A \circ e =  \mathbb{I} \circ e$, and $1$ is an eigenvalue corresponding to the eigenvector $e$. 
\end{proof}

The circle product also proves to be useful  in the multidimensional matrix equation for perfect colorings of hypergraphs~\cite{my.hyperperfcolor}.  With the help of the circle product  we also define the covering  of matrices that can be considered as the inverse operation of the Kronecker product.

Let $A$ be a $d$-dimensional matrix of order $n_1$ and $B$ be a $d$-dimensional matrix of order $n_2$,  where $n_2 \leq n_1$. We will say that the matrix $A$ \textit{covers} the matrix $B$ ($A$ is a \textit{covering} of $B$) if there exists a rectangular $(0,1)$-matrix $P$ of sizes $n_1 \times n_2$ such that $A \circ P = P \circ B$ and every row of $P$ contains exactly one $1$.

\begin{utv}
Let $A$ be a $d$-dimensional matrix of order $n_1$. Then for every $n_2 \in \mathbb{N}$ matrices $ \frac{1}{n_2^{d-1}}(A \otimes  J^{d}_{n_2})$   and $\frac{1}{n_2}(A \otimes  \mathbb{I})$ cover the matrix $A$. Here $\mathbb{I}$ is the $d$-dimensional identity $(0,1)$-matrix of order $n$ that  has ones only at the main diagonal.  
\end{utv}

\begin{proof}
To obtain a covering for both cases it is sufficient to take the matrix $P$ of sizes $(n_1n_2) \times n_2$ such that  $p_{i,j} =1$ if and only if $j = \lceil \frac{i}{n_2} \rceil$.
\end{proof}

Let us turn to the connection between products of matrices and multiary quasigroups. 

A \textit{$d$-ary quasigroup} $f$ of order $n$ is a function $f: I_n^d \rightarrow I_n$ such that the equation $x_0 = f(x_1, \ldots, x_d)$ has a unique solution for any one variable if all the other $d$ variables are specified arbitrarily.  The Cayley table of a $d$-ary quasigroup of order $n$ is a $d$-dimensional latin hypercube of order $n$, and vice versa, every $d$-dimensional latin hypercube can be considered as the Cayley table of some $d$-ary quasigroup. The graph $$\left\{(x_0, x_1, \ldots, x_d) |~ x_0 =f(x_1, \ldots, x_d) \right\}$$ of a quasigroup $f$ is the set of ones of the $(d+1)$-dimensional permutation $M^f$ of order $n$. 

For $d$-ary quasigroups we define \textit{transversals}  so that they coincide with those in the corresponding latin hypercubes.   Then  every transversal in  a $d$-ary quasigroup $f$ is a positive diagonal in the $(d+1)$-dimensional permutation matrix $M^f$, so the number of transversals in $f$ is equal to the permanent of  $M^f$.

Let $f$ be a $d_1$-ary quasigroup of order $n$ and $g$ be  a $d_2$-ary quasigroup of order $n$.  The \textit{composition}  $f \cdot g$ of quasigroups $f$ and $g$ is the $(d_1 + d_2 -1)$-ary quasigroup of order $n$ defined as
$$(f \cdot g) (x_1, \ldots, x_{d_1 + d_2 - 1}) = x_{0}  \Leftrightarrow f (x_{1}, \ldots, x_{d_1 - 1}, g (x_{d_1}, \ldots, x_{d_1+ d_2 -1}) ) = x_0.$$

\begin{utv}
Let $f$ be a $d_1$-ary quasigroup of order $n$ and $g$ be  a $d_2$-ary quasigroup of order $n$.  Suppose that   $M^f$ and $M^g$ are the multidimensional permutations  corresponding to quasigroups $f$ and $g$. Then $M^f  M^g$ is the multidimensional permutation corresponding to the composition $f \cdot g$.
\end{utv}

\begin{proof}
By the definition, an entry $c_{\gamma}$ of the matrix $C = M^f M^g$ is equal to $c_{\gamma} = \sum\limits_{i = 1}^n m^f_{\alpha i} m^g_{i\beta}$, where $\gamma = \alpha \beta$, $\alpha = (x_0, x_1, \ldots, x_{d_1 -1})$, and $\beta = (x_{d_1}, \ldots, x_{d_1 + d_2 -1})$.

Since every line of matrices $M^f$ and $M^g$ contains exactly one $1$, an entry $c_\gamma = 1$ if $m^f_{\alpha k}  = m^g_{k \beta} = 1$ for some $k \in \{1, \ldots, n \}$,  and $c_\gamma = 0$ otherwise. Equivalently, we have that $c_\gamma = 1$ if and only if $f(x_{1}, \ldots, x_{d_1 - 1}, k) = x_0$ and $g(x_{d_1}, \ldots, x_{d_1+ d_2 -1})  = k$ for some  $k \in \{1, \ldots, n \}$. It remains to note that the matrix $C$ is exactly the graph of the composition $f \cdot g$. 
\end{proof}

The proposition also implies that  the graph of a consecutive composition of a series quasigroups can be obtained as an appropriate dot product their multidimensional permutations. 
Moreover, the following result from~\cite[Proposition 12]{my.obzor} is a direct corollary of the inequality on the dot product of nonnegative matrices that will be stated as Theorem~\ref{dotperm} below.

\begin{utv}[\cite{my.obzor}]
Let $f$ be a $d_1$-ary quasigroup  of order $n$ and $g$ be  a $d_2$-ary quasigroup of order $n$.  Assume that $f$ and $g$ have $T(f)$ and $T(g)$ transversals, respectively. Then the number of transversals in the composition  $ f \cdot g$ is at least $T(f) T(g)$. 
\end{utv}

Let $f$ be a $d$-ary quasigroup of order $n_1$ and $g$ be a $d$-ary quasigroup of order $n_2$. 
Define the \textit{direct product}  $f \times g$ of quasigroups $f$ and $g$ as the $d$-ary quasigroup of order $n_1 n_2$ over the set $I_{n_1 n_2} = I_{n_1} \times I_{n_2}$ as
$$(f \times g) ((x_1, y_1), \ldots, (x_d, y_d)) = (x_{0}, y_0)  \Leftrightarrow f (x_{1}, \ldots, x_d) = x_0 \mbox{ and } g(y_1, \ldots, y_d) = y_0.$$

\begin{utv}
Let $f$ be a $d$-ary quasigroup of order $n_1$ and $g$ be a $d$-ary quasigroup of order $n_2$.  Suppose that   $M^f$ and $M^g$ are the multidimensional permutations  corresponding to  $f$ and $g$. Then $M^f  \otimes  M^g$ is the multidimensional permutation corresponding to the direct product $f \times g$.
\end{utv}

\begin{proof}
By the definition, an entry $c_{\gamma}$ of the matrix $C = M^f  \otimes M^g$ is equal to $c_{\gamma} =   m^f_{\alpha} m^g_{\beta}$, where $\gamma_i = (\alpha_i -1)n_2 + \beta_i$ for each $i = 1, \ldots, d$,  $\alpha = (x_0, x_1, \ldots, x_{d_1})$, and $\beta = (y_0, y_1, \ldots, y_d)$.

Note that  an  entry $c_\gamma = 1$ if $m^f_{\alpha}  = m^g_{\beta} = 1$,  and $c_\gamma = 0$ otherwise. Equivalently, $c_\gamma = 1$ if and only if $f( x_1, \ldots, x_{d}) = x_0$ and $g ( y_1, \ldots, y_d ) = y_0$. It remains to note  that the matrix $C$ is exactly the graph of $f \times g$. 
\end{proof}

\section{Products of stochastic matrices} \label{stocprodsec}

In this section, we study the products of stochastic matrices. We start with the outer product.

\begin{teorema} \label{outerstoch}
Let $A$ be a $d_1$-dimensional $k_1$-stochastic matrix of order $n$ and $B$  be a $d_2$-dimensional $k_2$-stochastic matrix of order $n$. Then $n^{ k_1 + k_2 -r}  (A \times B)$ is a $(d_1 + d_2)$-dimensional $r$-stochastic matrix, where $r = \max \{ d_1 + k_2, d_2 + k_1\}$. 
\end{teorema}  

\begin{proof}
Let us choose an $r$-dimensional plane $\Gamma$ in the matrix $C = A \times B$. Without loss of generality, we assume that the first $r_1$ components and the last $r_2$ components of indices vary in $\Gamma$, $r_1 + r_2 = r$. Since $r  = \max \{ d_1 + k_2, d_2 + k_1\}$, we have  $r_1 \geq k_1$ and $r_2 \geq k_2$. Then 
\begin{gather*}
\sum \limits_{\gamma \in \Gamma} c_\gamma = \sum\limits_{i_1, \ldots ,i_{r_1} = 1}^n \sum\limits_{j_{d_2 - r_2 +1}, \ldots, j_{d_2} =1}^n  a_{i_1, \ldots ,i_{r_1}, \alpha} b_{\beta, j_{d_2 - r_2 +1}, \ldots, j_{d_2}}  = \\
 \sum\limits_{i_1, \ldots ,i_{r_1} = 1}^n  a_{i_1, \ldots ,i_{r_1}, \alpha}  \sum\limits_{j_{d_2 - r_2 +1}, \ldots, j_{d_2} =1}^n b_{\beta, j_{d_2 - r_2 +1}, \ldots, j_{d_2}}  = n^{r_1 - k_1} n^{r_2 - k_2} = n^{r - k_1 -k_2},
 \end{gather*}
 where index $\gamma = (i_1, \ldots, i_{r_1}, \alpha, \beta, j_{d_2 - r_2 +1}, \ldots, j_{d_2})$.
Thus, the matrix $n^{ k_1 + k_2 -r}C$ is  $r$-stochastic.
\end{proof}

Next, we consider the Kronecker product of stochastic matrices. Note that one of the basic constructions of orthogonal arrays from~\cite{MacHar.eulersq} can be considered as the Kronecker product of the corresponding stochastic matrices.

\begin{teorema}
Let $A$ be a $d$-dimensional $k_1$-stochastic matrix of order $n_1$ and $B$  be a $d$-dimensional $k_2$-stochastic matrix of order $n_2$. Then $n^{k_1 + k_2 - 2r} (A \otimes B)$ is a $d$-dimensional $r$-stochastic matrix, where  $r = \max \{ k_1, k_2\}$.  In particular, the Kronecker product of polystochastic matrices is polystochastic.
\end{teorema}

\begin{proof}
Let us choose an $r$-dimensional plane $\Gamma$ in the matrix $C = A \otimes B$. Without loss of generality, we assume that the first $r$ components of the indices vary  in $\Gamma$. Then 
\begin{gather*}
\sum \limits_{\gamma \in \Gamma} c_\gamma = \sum\limits_{i_1, \ldots ,i_{r} = 1}^n  \sum\limits_{j_1, \ldots ,j_{r} = 1}^n  a_{i_1, \ldots ,i_{r}, \alpha} b_{j_1, \ldots , j_{r}, \beta}  = \\
 \sum\limits_{i_1, \ldots ,i_{r} = 1}^n   a_{i_1, \ldots ,i_{r}, \alpha}   \sum\limits_{j_1, \ldots ,j_{r} = 1}^n b_{j_1, \ldots , j_{r}, \beta}  = n^{r - k_1} n^{r - k_2} = n^{2r - k_1 -k_2},
 \end{gather*}
 where the index $\gamma$ is such that  $\gamma_{\ell} = (i_\ell - 1) n_2 + j_{\ell}$ for $\ell \in \{ 1, \ldots, r\}$ and $\gamma_\ell = (\alpha_\ell -1)n_2 +\beta_\ell $ for $\ell \in \{r+1, \ldots, d \}$.
Thus the matrix $n^{k_1 + k_2 - 2r} C$ is  $r$-stochastic.
\end{proof}

The contraction and projection  preserve or strengthen the degree of stochastivity.

\begin{teorema} \label{contractstoch}
Let $A$ be a $d$-dimensional $k$-stochastic matrix of order $n$ and $S \subseteq \{ 1, \ldots, d\}$, $|S| = \ell$. If $k + \ell \leq d$, then the contraction $\frac{1}{n}A_S$ is a $(d - \ell)$-dimensional $k$-stochastic matrix of order $n$.
\end{teorema}

\begin{proof}
Let us choose a $k$-dimensional plane $\Gamma$ in the matrix $C = A_S$. Without loss of generality, we assume that $S = \{d - \ell +1, \ldots, d \}$ and that  the first $k$ components of the indices vary  in  the plane $\Gamma$. Then 
\begin{gather*}
\sum \limits_{\gamma \in \Gamma} c_\gamma = \sum\limits_{i_1, \ldots ,i_{k} = 1}^n  \sum\limits_{j =1}^n  a_{i_1, \ldots ,i_{k}, \alpha, j, \ldots, j}  = 
\sum\limits_{j =1}^n \sum\limits_{i_1, \ldots ,i_{k} = 1}^n  a_{i_1, \ldots ,i_{k}, \alpha, j, \ldots, j}   = n,
 \end{gather*}
 where the index $\gamma = (i_1, \ldots, i_k, \alpha)$ and $\alpha \in I_{n}^{d - k - \ell}$.
Thus, the matrix $\frac{1}{n} C$ is  $k$-stochastic.
\end{proof}

\begin{teorema}
Let $A$ be a $d$-dimensional $k$-stochastic matrix of order $n$ and $S \subseteq \{ 1, \ldots, d\}$, $|S| = \ell$. Then the projection $ P_S(A)$ is a $(d - \ell)$-dimensional $(k - \ell)$-stochastic matrix of order $n$ if $k > \ell$, and $P_S(A) = n^{\ell - k+1} J_n^{d-\ell}$ if $k \leq \ell$. 
\end{teorema}

\begin{proof}
Assume that $k > \ell$.
Let us choose a $(k - \ell)$-dimensional plane $\Gamma$ in the matrix $C = P_S(A)$. Without loss of generality, we assume that $S = \{d - \ell +1, \ldots, d \}$ and that  the first $k - \ell$ components of the indices vary  in $\Gamma$. Then 
\begin{gather*}
\sum \limits_{\gamma \in \Gamma} c_\gamma = \sum\limits_{i_1, \ldots ,i_{k - \ell} = 1}^n  \sum\limits_{j_1, \ldots, j_\ell =1}^n  a_{i_1, \ldots ,i_{k-\ell}, \alpha, j_1, \ldots, j_\ell}  = 1,
 \end{gather*}
 where the index $\gamma = (i_1, \ldots, i_{k-\ell}, \alpha)$.
Thus, the matrix $C$ is  $(k - \ell)$-stochastic.

If $k \leq \ell$, then every element of the matrix $C = P_S(A)$ is
\begin{gather*}
 c_\gamma =  \sum\limits_{j_1, \ldots, j_\ell =1}^n  a_{\gamma, j_1, \ldots, j_\ell}  = n^{\ell - k}.
 \end{gather*}
 Thus, $P_S(A)$ is equal to  $n^{\ell - k+1} J_n^{d-\ell}$.
\end{proof}

Let us now turn to  more complex matrix operations. 
The dot product of stochastic matrices  has already been studied in paper~\cite{ChildWan.multper}. Meanwhile,  Lemma 3.4 from~\cite{ChildWan.multper}  contains a   mistake that we correct here.

\begin{teorema} \label{dotstoch}
Let $A$ be a $d_1$-dimensional $k_1$-stochastic matrix of order $n$,  $B$  be a $d_2$-dimensional $k_2$-stochastic matrix of order $n$, and $r = \max \{  k_1 + d_2, k_2 + d_1\} $. If $r \leq d_1 + d_2 -2$, then $n^{k_1 + k_2 - r -1} (AB)$ is a $(d_1 + d_2 -2)$-dimensional  $r$-stochastic matrix of order $n$. Moreover, if $A$ and $B$ are polystochastic matrices, then $AB$ is a polystochastic matrix. 
\end{teorema}  
\begin{proof}
Let $k_1 > 1$ or $k_2 > 1$.  By the definition, $AB$ is the $(d_1, 1)$-contraction of $A \times B$. By Theorem~\ref{outerstoch}, $n^{k_1 + k_2 - r} (A \times B)$ is an $r$-stochastic matrix with $r = \max\{ d_1 +k_2, d_2 + k_1 \}$. Using Theorem~\ref{contractstoch} for the $(d_1,1)$-contraction, we obtain that $ n^{k_1 + k_2 - r -1} (AB)$ is a $(d_1 + d_2 -2)$-dimensional  $r$-stochastic matrix of order $n$.

Assume now that  $k_1 = k_2  = 1$. Without loss of generality, consider the $1$-dimensional plane $\Gamma$ in which the first component of indices vary.  Then for the matrix $C = AB$ we have
\begin{gather*}
\sum \limits_{\gamma \in \Gamma} c_\gamma =  \sum\limits_{i = 1}^n \sum\limits_{j = 1 }^n   a_{i,  \alpha, j } b_{j, \beta}  = 
\sum\limits_{j = 1}^n    b_{j, \beta}  \sum\limits_{i = 1}^n   a_{i,  \alpha, j }  =1,
 \end{gather*}
 where index $\gamma = (i, \alpha, \beta)$, $\alpha \in I_n^{d_1 -2}$, $\beta \in I_n^{d_2-1}$.
Thus, the matrix $C$ is  polystochastic.
\end{proof}

Lemma 3.4 from~\cite{ChildWan.multper} claimed that the dot product of $k_1$- and $k_2$-stochastic matrices $A$ and $B$ is an $r$-stochastic matrix with $r = \max \{  k_1 + d_2, k_2 + d_1\} -2$. It is not hard to find a pair of $2$-dimensional $2$-stochastic matrices $A$ and $B$ (in which the sums of all entries are equal to $1$) satisfying $AB = 0$ that will be a counterexample to this statement. 

The following result on the dot product of a polystochastic matrix and the uniform matrix can be proved by a direct calculation.

\begin{utv} \label{dotuniform}
If $A$ is a $d$-dimensional polystochastic matrix of order $n$, then
$$A  J^t_n = J^{d+t-2}_n \mbox { and } J^t_n A = J^{d+t-2}_n.$$ 
\end{utv}

As a corollary from Theorem~\ref{dotstoch}, we have the following property of  the circle product.

\begin{teorema}
Let $A$ be a $d_1$-dimensional polystochastic  matrix of order $n$ and $B$  be a $d_2$-dimensional polystochastic matrix of order $n$. Then $A \circ B$ is a polystochastic matrix. 
\end{teorema}  

\begin{proof}
 By the definition, $A \circ B$ is the $(d_1 -1)$-times dot product:
  $$A \circ B = (\cdots((A \cdot_{d_1, 1}B) \cdot_{d_1 - 1, 1}B) \cdots) \cdot_{2, 1}B.$$
 
 By Theorem~\ref{dotstoch}, if $A$ and $B$ are polystochastic matrices, then the dot product of $A$ and $B$ is a polystochastic matrix. So in this case, the matrix $A \circ B$ is also polystochastic. 
\end{proof}

If $A$ is a $k_1$-stochastic matrix and $B$ is a $k_2$-stochastic matrix, where $k_1$ or $k_2$ is greater than $1$, then  $A \circ B$ is not necessarily a $k$-stochastic matrix. By Theorem~\ref{dotstoch},  for $d_1 \geq 2$ each of the $d_1 - 1$ dot products in the circle product  $A \circ B$ decreases the difference between the degree of stochastivity $r$ and the dimension $d$ of the matrices until the former exceeds the latter. 

Since $A \circ B$ is an appropriate dot product of matrices $A$ and $B$, the following is also straightforward from Proposition~\ref{dotuniform}.

\begin{utv} 
If $A$ is a $d$-dimensional polystochastic matrix of order $n$, then
$$A \circ  J^t_n = J^{(d-1)(t-1)+1}_n  \mbox{ and } J^t_n \circ  A = J^{(d-1)(t-1)+1}_n.$$ 
\end{utv}

\section{Permanents  of products of multidimensional matrices} \label{perprodsec}

In this section we estimate the permanent of the above products of matrices by means of the permanent of the factors. We start with the outer product, for which we have the exact equality.

\begin{teorema} \label{directprodper}
Let  $A$ be a $d_1$-dimensional matrix of order $n$ and $B$ be a $d_2$-dimensional matrix of order $n$. Then 
$$\per (A \times B) = n!( \per A \cdot \per B).$$
\end{teorema}

\begin{proof}
Let $D_A = \{\alpha^1, \ldots, \alpha^n\}$ be a diagonal in the matrix $A$  and $D_B = \{ \beta^1, \ldots, \beta^n\}$  be a diagonal in the matrix $B$. For every permutation $\sigma \in S_n$, consider a collection of indices $D = \{ \gamma^1, \ldots, \gamma^n\}$ such that $\gamma^i = \alpha^i \beta^{\sigma(i)}$ for all $i \in \{ 1, \ldots, n\}$. From the definition of outer product,  we see that $D$ is a diagonal in the matrix $A \times B$. Moreover, every diagonal of $A \times B$ is  covered by this construction. 
 Therefore, for the matrix $C = A \times B$ we have
\begin{gather*}
\per C = \sum\limits_{D \in \mathcal{D}(A \times B)} \prod\limits_{\gamma \in D} c_\gamma  = n! \sum\limits_{D_A \in \mathcal{D}(A), D_B \in \mathcal{D}(B)} \prod\limits_{\alpha \in D_A}  a_{\alpha} \prod\limits_{\beta \in D_B} b_{\beta} \\
= n! \left( \sum\limits_{D_A \in \mathcal{D}(A)} \prod\limits_{\alpha \in D_A}  a_{\alpha}  \right) \left( \sum\limits_{D_B \in \mathcal{D}(B)} \prod\limits_{\beta \in D_B}  b_{\beta}  \right) =n! (\per A \cdot \per B). 
\end{gather*}
\end{proof}

As a corollary, we have that for every pair of matrices  $A$ and $B$ of the same order and dimensions $d_1$ and $d_2$,  there exists a matrix  of dimension $d_1 +d_2$ whose permanent is equal to $\per A \cdot \per B$. As it was shown~\cite{my.values01}, one can find a matrix of one less dimension satisfying this property. For the sake of completeness, we provide its proof here.

\begin{utv}[\cite{my.values01}]
For a $d_1$-dimensional matrix $A$ of order $n$ and a $d_2$-dimensional matrix $B$ of order $n$, there exists a $(d_1 + d_2 -1)$-dimensional matrix $C$ such that $\per C = \per A \cdot \per B$. 
\end{utv}

\begin{proof}
Let us fix some permutation $\sigma \in S_n$ and define the matrix $C$ with entries $c_{\gamma}$, $\gamma = i \alpha \beta$, as
$$c_{i \alpha \beta} = a_{i \alpha} b_{\sigma(i) \beta}.$$
In other words, the matrix $C$ is  composed by $n$  different $(d_1 + d_2 -2)$-dimensional planes of the matrix $A \times B$  of the same direction such that no  two of these planes  were in the same hyperplane of $A \times B$.   

Note that each pair of diagonals $D_A = \{(1, \alpha^1), \ldots, (n, \alpha^n)\}$ from the  matrix $A$  and $D_B = \{ (1, \beta^1), \ldots, (n,\beta^n)\}$ from the matrix $B$  one-to-one corresponds to a  diagonal $D = \{ \gamma^1, \ldots, \gamma^n\}$ of the matrix $C$, where  $\gamma^i = i \alpha^i  \beta^{\sigma(i)}$. 

Repeating the calculations from the proof of Theorem~\ref{directprodper}, we obtain $\per C = \per A \cdot \per B$.
\end{proof}

Next, we study the permanent of the Kronecker product of multidimensional matrices. In~\cite{brualdi.perdirprod}, Brualdi proved the following inequalities on the permanent of the Kronecker product of nonnegative $2$-dimensional matrices.

\begin{teorema}[\cite{brualdi.perdirprod}]
Let $A$ and $B$ be nonnegative $2$-dimensional matrices of orders $n_1$ and $n_2$, respectively. Then there exists a constant $K$, $K \geq 1$, depending only on  $n_1$ and $n_2$ such that 
$$ (\per A)^{n_2} (\per B)^{n_1} \leq \per (A \otimes B) \leq K (\per A)^{n_2} (\per B)^{n_1}.$$
\end{teorema}

Unfortunately, both of these inequalities do not hold for multidimensional matrices.

\begin{zam}
There exists a nonnegative $3$-dimensional matrix $A$ of order $n$ such that $\per (A \otimes A) < (\per A)^{2n}$.
\end{zam}

\begin{proof}
Consider the following $3$-dimensional matrix $A$ of order $n =3$
$$
A = \left( \begin{array}{ccc|ccc|ccc}
1 & 0 & 0 & 0 & 0 & 0 & 0 & 1 & 0 \\
0 & 0 & 1 & 0 & 1 & 0 & 0 & 0 & 0 \\
0 & 0 & 0 & 1 & 0 & 0 & 0 & 0 & 1 \\
\end{array}\right).
$$
It is easy to see that the permanent of $A$ is equal to $2$, so $(\per A)^{2n} = (\per A)^{6} = 64$. On the other hand, it can be checked by direct calculations that the permanent of $A \otimes A$ is $40$.  
\end{proof}

\begin{zam}
There exist nonnegative $3$-dimensional matrices $A$ and $B$ such that $\per B = 0$ but $\per (A \otimes B) >0$.
\end{zam}

\begin{proof}

Consider  the following $3$-dimensional matrices of order $2$:
$$  A = \left( \begin{array}{cc|cc}
1 & 1 & 1 & 1 \\
1 & 1 & 1 & 1 \\
\end{array} \right); ~~~
B = \left( \begin{array}{cc|cc}
0 & 1 & 1 & 0 \\
1 & 0 & 0 & 1 \\
\end{array} \right). 
$$

Note that $\per A =4$ and $\per B = 0$. Then 

$$ A \otimes B =
\left( \begin{array}{cccc|cccc|cccc|cccc}
0 & 1 & 0 & 1 &  0 & 1 & 0 & 1 & 1 & 0 & 1 & 0 &  1 & 0 & 1 & 0 \\
1 & 0 & 1 & 0 &  1 & 0 & 1 & 0 & 0 & 1 & 0 & 1 &  0 & 1 & 0 & 1 \\
0 & 1 & 0 & 1 &  0 & 1 & 0 & 1 & 1 & 0 & 1 & 0 &  1 & 0 & 1 & 0 \\
1 & 0 & 1 & 0 &  1 & 0 & 1 & 0 & 0 & 1 & 0 & 1 &  0 & 1 & 0 & 1 \\
\end{array} \right)
$$
and it can be checked that $\per (A \otimes B) = 64 > 0$.
\end{proof}

Meanwhile, the following lower bound on the Kronecker product of nonnegative matrices remains true for the multidimensional case.

\begin{teorema}
Let $A$ and $B$ be nonnegative $d$-dimensional matrices of orders $n_1$ and $n_2$, respectively. Then
$$\per (A \otimes B) \geq  \per A (\per B)^{n_1} +  (\per A)^{n_2} \per B - \per A \per B.$$
\end{teorema}

\begin{proof}
By the definition, the Kronecker product $C = A \otimes B$ has entries $c_{\gamma} = a_{\alpha} b_{\beta}$, where $\gamma_i = (\alpha_i -1)n_2 + \beta_i$ for each $i = 1, \ldots, d$.  

Let  us fix a positive diagonal $D_A = \{ \alpha^1, \ldots, \alpha^{n_1}\}$ in the matrix $A$. To every collection of $n_1$ positive diagonals $D^1_B = \{ \beta^{1,1},\ldots,  \beta^{n_2,1} \}$, \ldots, $D^{n_1}_B = \{ \beta^{1,n_1}, \ldots, \beta^{n_2,n_1} \}$ of the matrix $B$,  we put into a correspondence a set of indices $D = \{ \gamma^1, \ldots, \gamma^{n_1 n_2} \}$ in the matrix $A \otimes B$, where for each $m \in \{ 1, \ldots, n_1n_2\}$, $m = (t- 1) n_2 + r$, $1 \leq r \leq n_2$, and $i \in \{ 1, \ldots, d\} $, we take  $\gamma^m_i = (\alpha^{t}_i -1)n_2 + \beta^{r,t}_i $.

It can be checked that the set $D$ is a diagonal in the Kronecker product $C$.
Since  $c_{\gamma^m} = a_{\alpha^t} b_{\beta^{r,t}}$ and all $a_{\alpha^t}$ and $b_{\beta^{r,t}}$ are nonzero, a diagonal $D$ is also positive. So we proved that $\per (A \otimes B) \geq \per A (\per B)^{n_1}$.

Acting similarly, one can deduce that $\per (A \otimes B) \geq (\per A)^{n_2} \per B$. Note that  these two constructions produce different diagonals except of diagonals $D = \{ \gamma^1, \ldots, \gamma^{n_1 n_2} \}$ such that  $\gamma^m_i = (\alpha^{t}_i -1)n_2 + \beta^{r}_i $ for some positive diagonals $D_A = \{ \alpha^1, \ldots, \alpha^{n_1}\}$  and $D_B = \{ \beta^{1},  \ldots, \beta^{n_2} \}$. Such diagonals $D$   arise exactly once  in both constructions. Therefore, 
$$\per (A \otimes B) \geq  \per A (\per B)^{n_1} +  (\per A)^{n_2} \per B - \per A \per B.$$
\end{proof}

For  the permanent of the projection of a nonnegative matrix we also have a simple lower bound.

\begin{teorema}
Let $A$ be a nonnegative $d$-dimensional matrix of order $n$ and $S \subseteq \{1, \ldots, d \}$. Then $\per P_S(A) \geq \per A$.
\end{teorema}

\begin{proof}
Without loss of generality, we assume that $S = \{1, \ldots, \ell \} $. 
Denote $C = P_S(A)$. By the definitions of the permanent and the projection, we have
\begin{gather*}
\per C = \sum\limits_{D_C \in \mathcal{D}(C)} \prod\limits_{\gamma \in D_C} c_{\gamma} = 
\sum\limits_{D_C \in \mathcal{D}(C)} \prod\limits_{\gamma \in D_C} \left(\sum\limits_{i_1, \ldots, i_{\ell} = 1}^n a_{i_1, \ldots, i_\ell, \gamma} \right).
\end{gather*} 
To each diagonal $D_A \in \mathcal{D}(A)$, $D_A = \{ (i^1_1, \ldots,i^1_\ell, \gamma^1 ) , \ldots, (i^n_1, \ldots,i^n_\ell, \gamma^n ) \}$ we put into a correspondence a diagonal  $D_C = \{  \gamma^1  , \ldots,  \gamma^n  \}$. So in $\per C$  every $\prod\limits_{\alpha\in D_A} a_{\alpha}$ is contained as a summand. Since the matrix $A$ is nonnegative, all other summands in $\per C$ are also nonnegative. Therefore, $\per C = \per P_S(A) \geq \per A$.  
\end{proof}

Note that the permanent of a projection cannot be estimated from the above by the permanent of a matrix.

\begin{zam}
There exists nonnegative $d$-dimensional matrices $A$ of order $n$ such that $\per A = 0$ but $\per P_S (A) >0$.
\end{zam}

\begin{proof}
It is sufficient to consider a $d$-dimensional $(0,1)$-matrix $A$ of order $n$ such that $a_{\alpha} = 1$ if and only if  $\alpha$ belongs to a fixed hyperplane $\Gamma$ of the first direction (direction $(1, 0, \ldots, 0)$). Then the $1$-projection of $A$ is  the matrix $n J^{d-1}_n$, whose permanent is equal to $(n!)^{d-2}$. 
\end{proof}

For the contraction of matrices (if it does not coincide with the projection), it is also not possible to bound the permanent by means of the permanent of the initial matrix.

\begin{zam}
\item
\begin{enumerate}
\item There exists a $d$-dimensional nonnegative matrix $A$ of order $n$ such that $\per A = 0$ but $\per A_S > 0 $ for some $S \subseteq \{1, \ldots, d \}$, $|S| \geq 2$.
\item  There exists a $d$-dimensional nonnegative matrix $A$ of order $n$ such that $\per A > 0$ but $\per A_S = 0 $ for some $S \subseteq \{1, \ldots, d \}$, $|S| \geq 2$.
\end{enumerate}
\end{zam}

\begin{proof}
1. Consider a $d$-dimensional $(0,1)$-matrix $A$ of order $n$ such that $a_{\alpha} = 1$ if and only if  $\alpha$ belongs to a fixed hyperplane $\Gamma$ of direction $(1, 0, \ldots, 0)$. Then for every subset $S$ such that  $|S| = \ell \geq 2$ and  $1 \in S$, the $S$-contraction of $A$ is  the matrix $n  J^{d-\ell}_n$, whose permanent is equal to $(n!)^{d-\ell-1}$. 

2. Consider a $d$-dimensional $(0,1)$-matrix $A$, whose all ones are located at the diagonal 
$$D = \{(1, 2, \ldots, 2), (2, 3,  \ldots, 3), \ldots, (n, 1, \ldots, 1) \}.$$
Obviously, $\per A = 1$. On the other hand, for  every subset $S$ such that  $|S| = \ell \geq 2$ and $1 \in S$, the $S$-contraction of $A$ is the $(d - \ell)$-dimensional zero matrix, whose permanent is $0$. 
\end{proof}

At last, let us estimate the permanent of the dot product of multidimensional matrices. For $2$-dimensional matrices, a similar result was proved in~\cite{brualdi.perdirprod}.

\begin{teorema} \label{dotperm}
Let $A$ be a nonnegative $d_1$-dimensional matrix of order $n$ and $B$ be a nonnegative  $d_2$-dimensional matrix of order $n$. Then
$$\per (A B) \geq \per A \cdot \per B.$$
\end{teorema}

\begin{proof}
By the definition of the permanent,
$$\per A \cdot \per B = \left(\sum\limits_{D_A \in \mathcal{D}(A)}  \prod\limits_{\alpha \in D_A} a_{\alpha} \right)  \left(\sum\limits_{D_B \in \mathcal{D}(B)}  \prod\limits_{\beta \in D_B} b_{\beta} \right) =  \sum\limits_{D_A \in \mathcal{D}(A)}  \sum\limits_{D_B \in \mathcal{D}(B)}   \prod\limits_{\alpha \in D_A}   \prod\limits_{\beta \in D_B}  a_{\alpha} b_{\beta}.$$

On the other hand, from the definition of the dot product, for the matrix $C = AB$ we have
$$\per (AB) = \sum\limits_{D_C \in \mathcal{D} (C)} \prod\limits_{\gamma \in D_C} c_\gamma = \sum\limits_{D_C \in \mathcal{D} (C)} \prod\limits_{\gamma \in D_C}  \left(  \sum\limits_{i=1}^n a_{\overline{\alpha},i} b_{i,\overline{\beta}} \right),$$
where $\gamma = ( \overline{\alpha},  \overline{\beta} )$,  $\overline{\alpha} \in I_{n}^{d_1 -1}$,  and $\overline{\beta} \in I_{n}^{d_2-1}$.

 Let $D_A = \{  \alpha^1, \ldots, \alpha^n\}$ and  $D_B = \{  \beta^1, \ldots, \beta^n\}$ be a pair of diagonals in $A$ and $B$, respectively.   For every $j \in \{ 1, \ldots, n\}$, we present $\alpha^j $ as $ (\overline{\alpha^j}, j)$ and $\beta^j $ as $ (j, \overline{\beta^j})$.
Note that for each pair of diagonals $D_A$ and $D_B$ the summand $ \prod\limits_{\alpha \in D_A}   \prod\limits_{\beta \in D_B}  a_{\alpha} b_{\beta} $ in $\per A \cdot \per B$ appears in  the expansion of $\prod\limits_{\gamma \in D_C}  \left(  \sum\limits_{i=1}^n a_{\overline{\alpha},i} b_{i,\overline{\beta}} \right)$ in $\per (AB)$, where the diagonal $D_C = \{ \gamma^1, \ldots, \gamma^d\}$  is such that $\gamma^j = (\overline{\alpha^j}, \overline{\beta^j}) $ for all $j \in \{ 1, \ldots, n\}$. Since the matrices $A$ and $B$ are nonnegative, all other summands in $\per (AB)$ are also nonnegative. As a result,
$$\per (A B) \geq \per A \cdot \per B.$$
\end{proof}

As a corollary, we have the following inequality on the circle product of nonnegative matrices.

\begin{sled}
Let $A$ be a nonnegative $d_1$-dimensional matrix of order $n$ and $B$ be a nonnegative  $d_2$-dimensional matrix of order $n$. Then
$$\per (A \circ B) \geq \per A \cdot (\per B)^{d_1-1}.$$
\end{sled}

At last, we note that it is not possible to bound $\per (AB)$ by $\per A$ and $\per B$ from the above.

\begin{zam}
There exist nonnegative $2$-dimensional matrices $A$ and $B$ of order $n$ such that $\per A = \per B = 0$ but $\per (AB) >0$.
\end{zam}

\begin{proof}
Consider
$$
A = \left( 
\begin{array}{cccc}
1 & 0 & \cdots  & 0 \\
1 & 0 & \cdots &  0 \\
\vdots  & \vdots  & \ddots  & \vdots \\
1 & 0 & \cdots &  0 \\
\end{array}
\right) ;~~~
B = \left( 
\begin{array}{cccc}
1 & 1 & \cdots  & 1 \\
0 & 0 & \cdots &  0 \\
\vdots  & \vdots  & \ddots  & \vdots \\
0 & 0 & \cdots &  0 \\
\end{array}
\right).
$$
It is obvious that $\per A = \per B = 0$. On the other hand, $AB$ is the matrix $n J_n^2$, whose permanent is equal to $n!$. 
\end{proof}

\section*{Acknowledgements}
This work was funded by the Russian Science Foundation under grant No 22-21-00202, https:// rscf.ru/project/22-21-00202/.


\bib{brualdi.perdirprod}{article}{
  author={Brualdi, Richard A.},
  title={Permanent of the direct product of matrices},
  journal={Pacific J. Math.},
  volume={16},
  date={1966},
  pages={471--482},
  issn={0030-8730},
}

\bib{CheCveWei.tesorfunc}{article}{
  author={Che, Maolin},
  author={Cvetkovi\'{c}, Dragana S.},
  author={Wei, Yimin},
  title={Inequalities on generalized tensor functions with diagonalizable and symmetric positive definite tensors},
  journal={Stat. Optim. Inf. Comput.},
  volume={6},
  date={2018},
  number={4},
  pages={483--496},
  issn={2311-004X},
  doi={10.19139/soic.v6i4.599},
}

\bib{ChildWan.multper}{article}{
  author={Child, Billy},
  author={Wanless, Ian M.},
  title={Multidimensional permanents of polystochastic matrices},
  journal={Linear Algebra Appl.},
  volume={586},
  date={2020},
  pages={89--102},
  issn={0024-3795},
  doi={10.1016/j.laa.2019.10.008},
}

\bib{my.values01}{article}{
  author={Guterman, A. E.},
  author={Evseev, I. M.},
  author={Taranenko, A. A.},
  title={Values of the permanent function on multidimensional $(0,1)$-matrices},
  journal={Sib. Math. J.},
  volume={63},
  date={2022},
  number={2},
  pages={262--276},
  issn={0037-4466},
  doi={10.1134/s0037446622020057},
}

\bib{KilMar.facttensor}{article}{
  author={Kilmer, Misha E.},
  author={Martin, Carla D.},
  title={Factorization strategies for third-order tensors},
  journal={Linear Algebra Appl.},
  volume={435},
  date={2011},
  number={3},
  pages={641--658},
  issn={0024-3795},
  doi={10.1016/j.laa.2010.09.020},
}

\bib{lim.eigentensor}{article}{
  author={Lim, L.-H.},
  title={Singular values and eigenvalues of tensors: A variational approach},
  conference={ title={1st IEEE International Workshop}, },
  book={ series={Computational Advances in Multi-Sensor Adapative Processing}, publisher={IEEE, Piscataway, NJ}, },
  date={2005},
  pages={129--132},
}

\bib{MacHar.eulersq}{article}{
  author={MacNeish, Harris F.},
  title={Euler squares},
  journal={Ann. of Math. (2)},
  volume={23},
  date={1922},
  number={3},
  pages={221--227},
  issn={0003-486X},
  doi={10.2307/1967920},
}

\bib{qi.eigentensor}{article}{
  author={Qi, Liqun},
  title={Eigenvalues of a real supersymmetric tensor},
  journal={J. Symbolic Comput.},
  volume={40},
  date={2005},
  number={6},
  pages={1302--1324},
  issn={0747-7171},
  doi={10.1016/j.jsc.2005.05.007},
}

\bib{ryser.conj}{article}{
  author={Ryser, H. J.},
  title={Neuere probleme in der kombinatorik},
  conference={ title={Vortr\"age \"uber Kombinatorik Oberwolfach}, address={Mathematisches Forschungsinstitut Oberwolfach, July 24--29, 1967}, },
  pages={69--91},
}

\bib{shao.tensprod}{article}{
  author={Shao, Jia-Yu},
  title={A general product of tensors with applications},
  journal={Linear Algebra Appl.},
  volume={439},
  date={2013},
  number={8},
  pages={2350--2366},
  issn={0024-3795},
  doi={10.1016/j.laa.2013.07.010},
}

\bib{my.obzor}{article}{
  author={Taranenko, Anna A.},
  title={Permanents of multidimensional matrices: properties and applications},
  language={Russian, with English and Russian summaries},
  journal={Diskretn. Anal. Issled. Oper.},
  volume={23},
  date={2016},
  number={4},
  pages={35--101},
  issn={1560-7542},
  translation={ journal={J. Appl. Ind. Math.}, volume={10}, date={2016}, number={4}, pages={567--604}, issn={1990-4789}, },
  doi={10.1134/S1990478916040141},
}

\bib{my.per44}{article}{
  author={Taranenko, Anna A.},
  title={Positiveness of the permanent of $4$-dimensional polystochastic matrices of order $4$},
  journal={Discrete Appl. Math.},
  volume={276},
  date={2020},
  pages={161--165},
  doi={10.1016/j.dam. 2019.02.001},
}

\bib{my.hyperperfcolor}{article}{
  author={Taranenko, Anna A.},
  title={Perfect colorings of hypergraphs},
  eprint={https://arxiv.org/abs/2208.03447},
}

\bib{wanless.surv}{article}{
  author={Wanless, Ian M.},
  title={Transversals in Latin squares: a survey},
  conference={ title={Surveys in combinatorics 2011}, },
  book={ series={London Math. Soc. Lecture Note Ser.}, volume={392}, publisher={Cambridge Univ. Press, Cambridge}, },
  date={2011},
  pages={403--437},
}

\bib{YaoLiuBu.tensorprod}{article}{
  author={Yao, Hongmei},
  author={Liu, Lei},
  author={Bu, Changjiang},
  title={Tensors product and hyperdeterminant of boundary formats product},
  journal={Linear Algebra Appl.},
  volume={483},
  date={2015},
  pages={1--20},
  issn={0024-3795},
  doi={10.1016/j.laa.2015.05.022},
}



\begin{bibdiv}
    \begin{biblist}[\normalsize]
    \bibselect{biblio}
    \end{biblist}
    \end{bibdiv}
\end{document}